\documentclass[a4paper,11pt]{amsart}
\usepackage{amssymb,amsfonts,amsxtra,amscd,mathrsfs,setspace}
\usepackage{fullpage}
\theoremstyle{plain}
\newtheorem{theorem}{Theorem}[section]
\newtheorem{lemma}[theorem]{Lemma}
\newtheorem{cor}[theorem]{Corollary}

\theoremstyle{definition}
\newtheorem{defi}[theorem]{Definition}
\newtheorem{example}[theorem]{Example}
\theoremstyle{remark}
\newtheorem{rem}[theorem]{Remark}
\numberwithin{equation}{section}
\newcommand{\sym}[1]{\ensuremath{\mathbb{S}_{#1}}}
\newcommand{\uenv}[1]{\ensuremath{\mathcal{U}(#1)}}
\newcommand{\cuenv}[1]{\ensuremath{\widehat{\mathcal{U}}(#1)}}
\newcommand{\invlim}[2]{\ensuremath{\varprojlim_{#1} #2}}
\newcommand{\cotimes}{\ensuremath{\widehat{\otimes}}}

\DeclareMathOperator{\Der}{Der}
\DeclareMathOperator{\id}{id}

\begin{document}
\title{A Poincar\'e-Birkhoff-Witt Theorem for profinite pronilpotent Lie algebras}
\author{Alastair Hamilton}
\address{Department of Mathematics and Statistics, Texas Tech University, Lubbock, TX 79409-1042. USA.} \email{alastair.hamilton@ttu.edu}
\begin{abstract}
We prove a version of the Poincar\'e-Birkhoff-Witt Theorem for profinite pronilpotent Lie algebras in which their symmetric and universal enveloping algebras are replaced with appropriate formal analogues and discuss some immediate corollaries of this result.
\end{abstract}
\thanks{The work of the author is supported by the Simons Foundation, grant 279462.}
\maketitle

\section{Introduction}

The standard Poincar\'e-Birkhoff-Witt Theorem provides a linear bijection between the symmetric algebra $S(\mathfrak{g})$ and the universal enveloping algebra $\uenv{\mathfrak{g}}$ of a Lie algebra $\mathfrak{g}$. In this short article we present a formal analogue of this theorem that replaces the algebra of polynomials $S(\mathfrak{g})$ with an algebra of formal power series. This theorem applies to a class of Lie algebras that we refer to as \emph{profinite pronilpotent Lie algebras} as they are inverse limits of finite-dimensional nilpotent Lie algebras -- the name \emph{formal Lie algebra} would be more appropriate if the term `formal' did not already carry several different meanings. These Lie algebras carry a simple topology allowing us to talk about the convergence of power series type expressions within them. Relevant examples tend to be infinite-dimensional, occurring for example in the study of various homotopy algebraic structures in which one must keep track of an infinite hierarchy of multilinear operations \cite{ncgeom}. This may explain why this analogue of the Poincar\'e-Birkhoff-Witt Theorem does not appear to have been considered before elsewhere.

In this paper we work in the category of graded profinite vector spaces, which are inverse limits of finite-dimensional vector spaces. These spaces sometimes go under the heading of `linearly compact vector space', cf. \cite{lincom}. This is the category housing such objects as algebras of formal power series. We recall some basic facts about and constructions involving them in Section \ref{sec_profinite}. In Section \ref{sec_PBW} we prove a formal analogue of the Poincar\'e-Birkhoff-Witt Theorem which states that a certain completion of the universal enveloping algebra of a profinite pronilpotent Lie algebra is linearly homeomorphic to an algebra of formal power series. The proof will make use of the standard Poincar\'e-Birkhoff-Witt Theorem, which we remind the reader holds also in the graded context \cite{supPBW}. The pronilpotence condition plays an essential role in the proof, ensuring that certain expressions converge appropriately. In Section \ref{sec_cor} we discuss some immediate corollaries of this result.

\subsection*{Notation and conventions}

We work with graded vector spaces over a ground field of characteristic zero. Our convention will be to suppress any further use of the adjective `graded' with the understanding that everything takes place within the graded context, making appropriate use of the Koszul sign rule for instance in defining the symmetric algebra and the action of the symmetric group $\sym{n}$ on tensor products. The closure of a subset $W$ of a topological space $X$ will be denoted by $\overline{W}$.

\section{Profinite vector spaces and algebras} \label{sec_profinite}

In this section we describe the category of profinite vector spaces as well as various examples of the types of formal algebraic objects that live in this category.

\begin{defi} \label{def_profvs}
A \emph{profinite vector space} is a vector space,
\begin{equation} \label{eqn_invlim}
V=\invlim{\alpha}{V_\alpha}
\end{equation}
which is an inverse limit of finite-dimensional vector spaces $V_{\alpha}$. The inverse system is part of the data and we assume, for convenience, that the projections $V\to V_\alpha$ are surjective. The kernels of these projections provide a basis for the neighborhoods of zero in the inverse limit topology on $V$, which is complete\footnote{In the sense that every Cauchy filter converges.} and Hausdorff. A morphism of profinite vector spaces is a linear map that is \emph{continuous} in the inverse limit topology.
\end{defi}

\begin{lemma} \label{lem_extend}
Any continuous linear map $f:M\to W$ that maps a dense subspace $M$ of a profinite vector space $V$ to another profinite vector space $W$ has a unique continuous linear extension from $V$ to $W$.
\end{lemma}

\begin{proof}
This is possible as $W$ is Hausdorff and complete. The proof, while not particularly short, is straightforward and very similar to that of Theorem 5.1 of \cite{treves}. Hence we omit it.
\end{proof}

The category of profinite vector spaces contains arbitrary products (which are inverse limits of finite products) and \emph{completed} tensor products, which are defined by,
\begin{equation} \label{eqn_cotimes}
V\cotimes W:= \invlim{\alpha,\beta}{V_\alpha\otimes W_\beta}
\end{equation}
and which are distributive with respect to products. The inverse limit topology \eqref{eqn_invlim} on such products of profinite spaces coincides with the product topology. Note that $\cotimes$ and $\otimes$ agree on finite-dimensional vector spaces. The universal property of limits ensures the existence of a canonical embedding of the usual tensor product $V\otimes W$ inside $V\cotimes W$ as a dense subspace in the inverse limit topology. This allows us, given morphisms $f$ and $g$ of profinite vector spaces, to define $f\cotimes g$ using Lemma \ref{lem_extend} as the unique continuous extension of $f\otimes g$.

This definition allows us to define certain formal objects, such as algebras of formal power series.

\begin{defi}
Given a profinite vector space $V$, we define its completed tensor algebra by,
\[ \widehat{T}(V):=\prod_{n=0}^\infty V^{\cotimes n}. \]
\end{defi}

It follows from \eqref{eqn_cotimes} that
\begin{equation} \label{eqn_tenlim}
\widehat{T}(V)=\invlim{\alpha}{\widehat{T}(V_\alpha)}.
\end{equation}
The preceding discussion implies that the usual tensor algebra,
\[ T(V):=\bigoplus_{n=0}^\infty V^{\otimes n} \]
sits inside $\widehat{T}(V)$ as a dense subspace. The usual multiplication on $T(V)$ extends uniquely to a continuous map,
\begin{equation} \label{eqn_tenmul}
\widehat{T}(V)\cotimes\widehat{T}(V) \to \widehat{T}(V),
\end{equation}
turning the completed tensor algebra into an algebra of formal noncommutative power series.

Using again the universal property of limits we see that the action of every permutation $\sigma\in \sym{n}$ on $V^{\otimes n}$ extends continuously to $V^{\cotimes n}$. We denote the space of coinvariants by $V^{\cotimes n}/\sym{n}$ (which is a quotient by a closed subspace, the kernel of the map \eqref{eqn_invmap}). We will also need the subspace of invariants $(V^{\cotimes n})^{\sym{n}}$, which coincides with the closure of the usual subspace of invariants $(V^{\otimes n})^{\sym{n}}$ in $V^{\cotimes n}$. Coinvariants and invariants are of course isomorphic through the map;
\begin{equation} \label{eqn_invmap}
\Phi_n^{\Sigma}: V^{\cotimes n} \to V^{\cotimes n}, \qquad \Phi_n^{\Sigma}(w):=\sum_{\sigma\in\sym{n}}\sigma\cdot w.
\end{equation}
It follows from \eqref{eqn_cotimes} and this isomorphism that,
\begin{equation} \label{eqn_symlim}
\left(V^{\cotimes n}\right)^{\sym{n}}=\invlim{\alpha}{\left(V_\alpha^{\otimes n}\right)^{\sym{n}}} \quad\text{and}\quad V^{\cotimes n}/\sym{n}=\invlim{\alpha}{V_\alpha^{\otimes n}/\sym{n}},
\end{equation}
hence these are profinite vector spaces.

\begin{defi}
Given a profinite vector space $V$, we define its completed symmetric algebra by,
\[ \widehat{S}(V):=\prod_{n=0}^\infty V^{\cotimes n}/\sym{n}. \]
\end{defi}
Again, using the usual lifting property of limits we define a map,
\[ \widehat{S}(V)\cotimes\widehat{S}(V)\to\widehat{S}(V) \]
continuously extending the usual multiplication on the dense subspace $S(V)$, the usual symmetric algebra, turning $\widehat{S}(V)$ into a commutative algebra of formal power series.

If $\mathfrak{g}$ is any Lie algebra we may define a decreasing filtration;
\begin{equation} \label{eqn_filter}
F_p(\mathfrak{g}), \quad p\geq 1;
\end{equation}
where $F_p(\mathfrak{g})$ is the ideal spanned by finite sums of Lie-monomials in $\mathfrak{g}$ of degree at least $p$. Recall that $\mathfrak{g}$ is nilpotent if $F_p(\mathfrak{g})$ vanishes for some $p$.

\begin{defi}
A profinite pronilpotent Lie algebra $\mathfrak{g}$ is a Lie algebra which is an inverse limit (in the category of Lie algebras)
\[ \mathfrak{g}=\invlim{\alpha}{\mathfrak{g}_\alpha} \]
of a system of finite-dimensional nilpotent Lie algebras $\mathfrak{g}_\alpha$. A morphism of profinite pronilpotent Lie algebras is a \emph{continuous} morphism of the underlying Lie algebras.
\end{defi}

\begin{example}
Given a profinite vector space $V$, consider the subspace $L(V)$ of $T(V)$ that is linearly spanned by all Lie-monomials in $V$ (otherwise known as the free Lie algebra on $V$). We define the completed free Lie algebra $\widehat{L}(V)$ to be the closure of $L(V)$ inside $\widehat{T}(V)$. It follows from \eqref{eqn_tenlim} that,
\[ \widehat{L}(V)=\invlim{\alpha}{\widehat{L}(V_\alpha)}=\invlim{\alpha,p}{\left[\widehat{L}(V_\alpha)\middle/\overline{F_p\left(\widehat{L}(V_\alpha)\right)}\right]}. \]
\end{example}

\begin{example} \label{exm_vecfld}
Given a finite-dimensional vector space $V$, consider the Lie algebra $\Der(\widehat{S}(V))$ consisting of \emph{continuous} derivations of the algebra $\widehat{S}(V)$ of formal power series. Such derivations are otherwise known as \emph{formal vector fields}. This Lie algebra contains a Lie subalgebra consisting of those vector fields that vanish at the origin together with their first derivatives. This Lie subalgebra is a profinite pronilpotent Lie algebra. Similar examples may be easily obtained by replacing $\widehat{S}(V)$ with $\widehat{L}(V)$ or $\widehat{T}(V)$. These Lie algebras and their associated Maurer-Cartan moduli spaces were used in \cite{ncgeom} to describe various moduli spaces of homotopy algebraic structures.
\end{example}

We mention at this point that $\widehat{S}$, $\widehat{T}$ and $\widehat{L}$ are functors on the category of profinite vector spaces. We define now the functor on profinite pronilpotent Lie algebras that is the main subject of this paper. We recall that given a Lie algebra $\mathfrak{g}$, its usual enveloping algebra $\uenv{\mathfrak{g}}$ is the quotient of the tensor algebra $T(\mathfrak{g})$ by the two-sided ideal $J(\mathfrak{g})$ contained in $T(\mathfrak{g})$ generated by the relations,
\begin{equation} \label{eqn_envrel}
x\otimes y = (-1)^{|x||y|}y\otimes x + [x,y]; \quad x,y\in\mathfrak{g}.
\end{equation}

\begin{defi}
Given a profinite pronilpotent Lie algebra $\mathfrak{g}$, we define its completed universal enveloping algebra $\cuenv{\mathfrak{g}}$ as the quotient of the completed tensor algebra $\widehat{T}(\mathfrak{g})$ by the closure of the ideal $J(\mathfrak{g})$ inside $\widehat{T}(\mathfrak{g})$.
\end{defi}

\begin{example} \label{exm_ufreel}
\[ \cuenv{\widehat{L}(V)} = \widehat{T}(V). \]
The argument follows standard lines, cf. \cite{bbak}. Since the elements of $\widehat{L}(V)$ have no constant terms, the multiplication \eqref{eqn_tenmul} yields, by Lemma \ref{lem_extend}, a continuous map from $\widehat{T}(\widehat{L}(V))$ to $\widehat{T}(V)$. Since this map respects the relations \eqref{eqn_envrel} it factors through a map from $\cuenv{\widehat{L}(V)}$ to $\widehat{T}(V)$. The inverse map comes from applying the functor $\widehat{T}$ to the inclusion of $V$ into $\widehat{L}(V)$ and then passing to the quotient $\cuenv{\widehat{L}(V)}$. To check that the maps are mutually inverse it suffices to check on the dense subspaces $T(L(V))$ and $T(V)$, which is trivial.
\end{example}

\begin{example}
Another example comes from taking the completed universal enveloping algebra of a profinite pronilpotent Lie algebra that is derived from the example of formal vector fields discussed in Example \ref{exm_vecfld}. As is well-known; cf. \cite{gerlie}, \cite{operad} and \cite{stalie}; the pre-Lie structure on formal vector fields may be interpreted in the framework of operads as the grafting of the root of one tree onto another. In \cite{hamCK} an example of a profinite pronilpotent Lie algebra was constructed in which this pre-Lie structure corresponds to the graph insertion operation of Connes-Kreimer \cite{CKHopf}. There it was shown that its \emph{completed} universal enveloping algebra recovers the dual of the Connes-Kreimer Hopf algebra, which was defined in \cite{CKHopf}.
\end{example}

\section{The Poincar\'e-Birkhoff-Witt Theorem} \label{sec_PBW}

In this section we will state and prove a formal version of the Poincar\'e-Birkhoff-Witt Theorem for profinite pronilpotent Lie algebras. We will make repeated use of the standard Poincar\'e-Birkhoff-Witt Theorem for ordinary Lie algebras.

If $\mathfrak{g}$ is a profinite pronilpotent Lie algebra then we may use \eqref{eqn_invmap} to define a map,
\begin{equation} \label{eqn_PBWmap}
\Phi^{\Sigma}:\widehat{S}(\mathfrak{g})\to\cuenv{\mathfrak{g}}.
\end{equation}

\begin{theorem} \label{thm_PBWthm}
The map $\Phi^{\Sigma}$ is a linear homeomorphism from $\widehat{S}(\mathfrak{g})$ to $\cuenv{\mathfrak{g}}$, where the latter space carries the quotient topology induced by $\widehat{T}(\mathfrak{g})$.
\end{theorem}

\begin{proof}
Since $\mathfrak{g}$ is an inverse limit of finite-dimensional nilpotent Lie algebras $\mathfrak{g}_\alpha$, we begin by proving that \eqref{eqn_PBWmap} is a linear bijection for each $\mathfrak{g}_\alpha$. This is equivalent to showing that the completed tensor algebra $\widehat{T}(\mathfrak{g}_\alpha)$ splits as a direct sum,
\begin{equation} \label{eqn_tenspl}
\widehat{T}(\mathfrak{g}_\alpha)=\overline{\Sigma}(\mathfrak{g}_\alpha)\oplus\overline{J}(\mathfrak{g}_\alpha),
\end{equation}
where $\overline{J}(\mathfrak{g}_\alpha)$ is the closure of the ideal $J(\mathfrak{g}_\alpha)$ inside $\widehat{T}(\mathfrak{g}_\alpha)$ and
\[ \overline{\Sigma}(\mathfrak{g}_\alpha)=\prod_{n=0}^\infty (\mathfrak{g}_\alpha^{\otimes n})^{\sym{n}} \]
is the subspace of invariants of the symmetric groups, which coincides with the closure inside $\widehat{T}(\mathfrak{g}_\alpha)$ of the usual subspace of invariants,
\[ \Sigma(\mathfrak{g}_\alpha):=\bigoplus_{n=0}^\infty (\mathfrak{g}_\alpha^{\otimes n})^{\sym{n}}. \]
Note that since $\mathfrak{g}_\alpha$ is finite-dimensional and hence discrete, the topology on $\widehat{T}(\mathfrak{g}_\alpha)$ is just the usual adic topology.

It follows from the standard Poincar\'e-Birkhoff-Witt Theorem for $\mathfrak{g}_\alpha$ that the usual tensor algebra splits as,
\[ T(\mathfrak{g}_\alpha)=\Sigma(\mathfrak{g}_\alpha)\oplus J(\mathfrak{g}_\alpha). \]
We may therefore define projection maps,
\[ \pi_\alpha^\Sigma:T(\mathfrak{g}_\alpha)\to\Sigma(\mathfrak{g}_\alpha) \quad\text{and}\quad \pi_\alpha^J:T(\mathfrak{g}_\alpha)\to J(\mathfrak{g}_\alpha) \]
onto these subspaces. We will prove that these maps are continuous in the adic topology.

Consider the filtration $F_p(\mathfrak{g}_\alpha)$ defined by \eqref{eqn_filter}. We may specify a filtration on $\mathfrak{g}_\alpha^{\otimes n}$ by defining $F_p(\mathfrak{g}_\alpha^{\otimes n})$ to be the subspace of $\mathfrak{g}_\alpha^{\otimes n}$ that is spanned by finite sums of elements from the family of subspaces;
\[ F_{p_1}(\mathfrak{g}_\alpha)\otimes\ldots\otimes F_{p_n}(\mathfrak{g}_\alpha), \quad p_1 + \cdots + p_n \geq p. \]
This leads to a decreasing filtration of the tensor algebra;
\[ F_pT(\mathfrak{g}_\alpha):=\bigoplus_{n=0}^\infty F_p\left(\mathfrak{g}_\alpha^{\otimes n}\right), \quad p\geq 0. \]

Now observe that if the term on the left-hand side of the relation \eqref{eqn_envrel} lies in $F_pT(\mathfrak{g}_a)$ then the two terms on the right-hand side will as well. More generally, if $w\in F_p(\mathfrak{g}_\alpha^{\otimes n})$ and $\sigma\in\sym{n}$ is any permutation then we may write the difference between $w$ and $\sigma\cdot w$ as the sum of an element in $J(\mathfrak{g}_\alpha)$ with an element in $F_p(\mathfrak{g}_\alpha^{\otimes n-1})$. Hence we may write,
\[ w=\frac{1}{n!}\Phi_n^\Sigma(w) + \eta + v, \]
for some $\eta\in J(\mathfrak{g}_\alpha)$ and $v\in F_p(\mathfrak{g}_\alpha^{\otimes n-1})$. It follows by induction on $n$ that $\pi_\alpha^\Sigma$ maps $F_p(\mathfrak{g}_\alpha^{\otimes n})$ to,
\[ F_pT_{\leq n}(\mathfrak{g}_\alpha):=\bigoplus_{i=0}^n F_p\left(\mathfrak{g}_\alpha^{\otimes i}\right). \]

Now, since $\mathfrak{g}_\alpha$ is nilpotent, there exists a $k$ such that $F_{k+1}(\mathfrak{g}_\alpha)$ vanishes. It follows that $F_{ik+1}(\mathfrak{g}_\alpha^{\otimes i})$ also vanishes, for all $i\geq 0$. Since $\mathfrak{g}_\alpha^{\otimes j}$ is contained in $F_j(\mathfrak{g}_\alpha^{\otimes j})$ for every $j\geq 0$, it follows from the above that for all $r>0$ the projection $\pi_\alpha^\Sigma$ maps $\mathfrak{g}_\alpha^{\otimes nk+r}$ to,
\[ \bigoplus_{i=0}^{nk+r} F_{nk+r}\left(\mathfrak{g}_\alpha^{\otimes i}\right) = \bigoplus_{i=n+1}^{nk+r} F_{nk+r}\left(\mathfrak{g}_\alpha^{\otimes i}\right). \]
In other words, $\pi_\alpha^\Sigma$ maps terms in $T(\mathfrak{g}_\alpha)$ of order greater than $nk$ to terms of order greater than $n$, that is $\pi_\alpha^\Sigma$ is continuous in the adic topology.

It follows that $\pi_\alpha^J=\id-\pi_\alpha^\Sigma$ is also continuous and hence by Lemma \ref{lem_extend} these projection maps extend to continuous maps,
\begin{equation} \label{eqn_prjsys}
\bar{\pi}_\alpha^\Sigma:\widehat{T}(\mathfrak{g}_\alpha)\to\overline{\Sigma}(\mathfrak{g}_\alpha) \quad\text{and}\quad \bar{\pi}_\alpha^J:\widehat{T}(\mathfrak{g}_\alpha)\to\overline{J}(\mathfrak{g}_\alpha),
\end{equation}
such that $\id=\bar{\pi}_\alpha^\Sigma+\bar{\pi}_\alpha^J$. This verifies \eqref{eqn_tenspl}.

By \eqref{eqn_tenlim} the tensor algebra $\widehat{T}(\mathfrak{g})$ is the inverse limit of the algebras $\widehat{T}(\mathfrak{g}_\alpha)$. Since the bonding maps in this inverse system are induced by morphisms of Lie algebras, they respect the decomposition \eqref{eqn_tenspl}. It follows that both systems of mappings defined in \eqref{eqn_prjsys} lift to continuous maps $\bar{\pi}^\Sigma$ and $\bar{\pi}^J$ defined on the inverse limit $\widehat{T}(\mathfrak{g})$.

Consider the representation of $T(\mathfrak{g})$ as the direct sum of the invariants $\Sigma(\mathfrak{g})$ and the ideal $J(\mathfrak{g})$ which arises by applying the standard Poincar\'e-Birkhoff-Witt Theorem to $\mathfrak{g}$. It is a straightforward observation from the definitions that the maps $\bar{\pi}^\Sigma$ and $\bar{\pi}^J$ agree on $T(\mathfrak{g})$ with the projection maps onto these respective summands. It follows that we have defined maps
\begin{equation} \label{eqn_prjmps}
\bar{\pi}^\Sigma:\widehat{T}(\mathfrak{g})\to\overline{\Sigma}(\mathfrak{g}) \quad\text{and}\quad \bar{\pi}^J:\widehat{T}(\mathfrak{g})\to\overline{J}(\mathfrak{g})
\end{equation}
such that $\id=\bar{\pi}^\Sigma+\bar{\pi}^J$, where $\overline{\Sigma}(\mathfrak{g})$ is isomorphic to $\widehat{S}(\mathfrak{g})$ via \eqref{eqn_invmap} as previously discussed. Therefore $\widehat{T}(\mathfrak{g})$ splits as the direct sum of the invariant subspace $\overline{\Sigma}(\mathfrak{g})$ and the ideal $\overline{J}(\mathfrak{g})$. This proves that \eqref{eqn_PBWmap} is a linear bijection from $\widehat{S}(\mathfrak{g})$ to $\cuenv{\mathfrak{g}}$. The equivalence of the two respective topologies on $\widehat{S}(\mathfrak{g})$ and $\cuenv{\mathfrak{g}}$ follows from the fact that the projection maps \eqref{eqn_prjmps} onto these summands are continuous.
\end{proof}

\section{Corollaries} \label{sec_cor}

In this section we will explain a few natural corollaries of Theorem \ref{thm_PBWthm}. In what follows $\cuenv{\mathfrak{g}}$ carries the quotient topology induced by $\widehat{T}(\mathfrak{g})$.

\begin{cor}
The completed universal enveloping algebra $\cuenv{\mathfrak{g}}$ of a profinite pronilpotent Lie algebra $\mathfrak{g}$ is complete and contains $\uenv{\mathfrak{g}}$ as a dense subspace.
\end{cor}

\begin{proof}
From the point of view of Theorem \ref{thm_PBWthm} and the standard Poincar\'e-Birkhoff-Witt Theorem, the canonical map from $\uenv{\mathfrak{g}}$ to $\cuenv{\mathfrak{g}}$ corresponds to the canonical embedding of $S(\mathfrak{g})$ as a dense subspace of $\widehat{S}(\mathfrak{g})$. The latter space is complete as it carries the inverse limit topology described in Definition \ref{def_profvs}.
\end{proof}

\begin{cor}
If $\mathfrak{g}=\invlim{\alpha}{\mathfrak{g}_\alpha}$ is a profinite pronilpotent Lie algebra then $\cuenv{\mathfrak{g}}=\invlim{\alpha}{\cuenv{\mathfrak{g}_\alpha}}$.
\end{cor}

\begin{proof}
It follows from \eqref{eqn_symlim} that $\widehat{S}(\mathfrak{g})=\invlim{\alpha}{\widehat{S}(\mathfrak{g}_\alpha)}$. Now apply Theorem \ref{thm_PBWthm}.
\end{proof}

\begin{cor} \label{cor_canemb}
The canonical map from a profinite pronilpotent Lie algebra $\mathfrak{g}$ to its completed universal enveloping algebra $\cuenv{\mathfrak{g}}$ is an embedding (in both the topological and algebraic sense) of $\mathfrak{g}$ onto a topologically closed commutator Lie subalgebra of $\cuenv{\mathfrak{g}}$.
\end{cor}

\begin{proof}
From the viewpoint of Theorem \ref{thm_PBWthm} this map is the canonical inclusion of $\mathfrak{g}$ into $\widehat{S}(\mathfrak{g})$, which is a topological embedding of $\mathfrak{g}$ onto a closed subspace of $\widehat{S}(\mathfrak{g})$.
\end{proof}

Now we are in a position to describe the precise sense in which the completed free Lie algebra is free.

\begin{cor} \label{cor_extend}
Every continuous linear map from a profinite vector space $V$ to a profinite pronilpotent Lie algebra $\mathfrak{g}$ has a unique extension to a continuous Lie algebra morphism from $\widehat{L}(V)$ to $\mathfrak{g}$.
\end{cor}

\begin{proof}
Applying the functor $\widehat{T}$ to a continuous linear map $\psi$ from $V$ to $\mathfrak{g}$, then passing to the quotient $\cuenv{\mathfrak{g}}$ and restricting to the Lie subalgebra $\widehat{L}(V)$ yields a continuous morphism of commutator Lie algebras,
\[ \psi^{\sharp}:\widehat{L}(V)\to\cuenv{\mathfrak{g}}. \]
Consider the embedding $\iota:\mathfrak{g}\to\cuenv{\mathfrak{g}}$ described in Corollary \ref{cor_canemb}. It is apparent that $\psi^{\sharp}$ maps $L(V)$ to the embedded image $\iota(\mathfrak{g})$, which is a closed subspace of $\cuenv{\mathfrak{g}}$ by Corollary \ref{cor_canemb}; therefore $\psi^\sharp$ maps the closure $\widehat{L}(V)$ here as well. The desired extension is therefore given by the composition $\iota^{-1}\circ\psi^\sharp$. The uniqueness of this extension is apparent.
\end{proof}

\begin{rem}
It should be mentioned that Corollary \ref{cor_extend} may be proved without the aid of Theorem \ref{thm_PBWthm}. Applying this Corollary to the case $\mathfrak{g}:=\widehat{L}(W)$ yields a useful description of the morphisms of a $C_\infty$-algebra, cf. \cite{ncgeom}.
\end{rem}

\end{document}